\documentclass[12pt]{amsart}
\usepackage[matrix,arrow]{xy}
\calclayout
\newtheorem{dummy}{realdumb}[section]
\newtheorem{theorem}[dummy]{Theorem}
\newtheorem{lemma}[dummy]{Lemma}

\newtheorem{proposition}[dummy]{Proposition}
\theoremstyle{definition}               
\newtheorem{definition}[dummy]{Definition}
\newtheorem{conjecture}[dummy]{Conjecture}
\theoremstyle{remark}

\newtheorem{remark}[dummy]{Remark}
\newcommand{\bz}{\mathbb Z}

\newcommand{\br}{\mathbb R}

\newcommand{\bd}{\partial}

\DeclareMathOperator{\cw}{CW}

\begin{document}
\hyphenation{}
\title{Wall's Finiteness Obstruction}
\author{Erik Kj\ae r Pedersen}
\address{Department of Mathematical Sciences, University of Copenhagen, Denmark}
\maketitle
\section*{Introduction}
\addtocounter{section}{1}
The purpose of this note is to give a self contained description of Wall's finiteness obstruction \cite{w7,w6}. One aim is to provide motivation for beginning students in a course on Algebraic $K$-Theory. In a final section we shall give some examples of applications and open questions relating to finiteness obstructions.

\section{Definitions and statements of theorems}
Throughout this note we work in the category of compactly generated topological spaces.
\begin{definition} \label{findom}
A connected topological space $X$ is finitely dominated if there exists a finite $\cw$-complex 
$K$ and maps $X\xrightarrow{i}K$ and $K\xrightarrow{r} X$ so that $r\circ i$ is homotopic
 to the identity.
\end{definition}

\begin{theorem} \label{tonetwo}
Let $X$ be finitely dominated. Then
\begin{enumerate}
\item $X$ is homotopy equivalent to a finite dimensional $\cw$-complex.
\item $X$ is homotopy equivalent to an infinite dimensional $\cw$-complex
with finite skeleta.
\item $X\times S^1$ is homotopy equivalent to a finite $\cw$-complex.
\item There is an invariant $\sigma(X) \in \tilde K_0(\bz[\pi_1(X,x)])$ so that $X$ is
homotopy equivalent to a finite $\cw$-complex if and only if $\sigma(X)=0$.
\item For every element $\sigma \in \tilde K_0(\bz[\pi_1(X,x)])$ there exists a finitely dominated
space $Y$ so that $\sigma(Y)=\sigma$.
\end{enumerate}
\end{theorem}

From now on we will mostly abbreviate $\pi_1(X,x)$ to $\pi$. A key step in the proof is the following: 

\begin{theorem} \label{tonethree} A connected space $X$ is finitely dominated if and only if the following hold:
\begin{enumerate}
\item $X$ is homotopy equivalent to a $\cw$-complex.
\item $\pi_1(X,x)$ is finitely presented.
\item The singular chains of the universal cover $S_*(\tilde X)$ is chain homotopy equivalent as a chain complex of $\bz[\pi]$-modules to 
a finite length 
chain complex
\[ 0\to P_k\to P_{k-1}\to \ldots \to P_0\to 0\]
of finitely generated projective $\bz[\pi]$-modules
\end{enumerate}
\end{theorem}

\begin{remark} \label{ronefour}
Assuming the theorem above we may consider
\[ [P_*] = \sum(-1)^i[P_i] \in K_0(\bz[\pi])\]
By a standard argument this is a well defined invariant of the chain homotopy type of $P_*$, and hence a well defined invariant of $X$. The inclusion $\bz\to \bz[\pi]$ followed by the augmentation $\bz[\pi]\to \bz$ provides a splitting
\[ K_0(\bz[\pi]) = K_0(\bz) \oplus \tilde K_0(\bz[\pi])= \bz\oplus \tilde K_0(\bz[\pi])\]
and we may write
\[ [P_*] = (\chi(X), \sigma(X))\]
It is easy to see that $\chi(X)$ is the usual Euler characteristic of $X$, and $\sigma(X)$ is now by definition the finiteness obstruction.
Observe that $\sigma(X)=0$ if $X$ is homotopy equivalent to a finite $\cw$-complex.
\end{remark}

To prove one direction of Theorem \ref{tonethree} we apply tricks due to M. Mather \cite{mat1} and A. Ranicki \cite{ra1}.
Preparing for Mather's trick we use

\begin{definition} Let $X$ be a space and $f:X\to X$ a self map. The mapping torus $T(X,f)$
is defined by 
\[T(X,f) = X\times [0,1]/(x,1)\sim (f(x),0)\]
\end{definition} 

\begin{lemma} The map $T(f): T(X,f)\to T(X,f)$ sending $(x,s)$ to $(f(x),s)$ is homotopic to the identity
\end{lemma}
\begin{proof} A homotopy $H$ is provided by 
\[
H(x,s,t) = \begin{cases} (x,s+t) \qquad &\text{if  $s+t\le 1$}\\
(f(x),s+t-1) & \text{if  $s+t\ge 1$}
\end{cases}\]
\end{proof}
\begin{lemma} Let $X$ and $Y$ be topological spaces and $f:X\to Y$ and $g:Y\to X$ maps. Then
$T(X,g\circ f)$ and $T(Y,f\circ g)$ are homotopy equivalent spaces.
\end{lemma}
\begin{proof} $f$ and $g$ induce maps $T(X,g\circ f) \to T(Y,f\circ g)$ and $T(Y,f\circ g)$ and $T(X,g\circ f)$, and the composites
 are homotopic to the identities by the lemma above.
\end{proof}

\begin{lemma} If $f,g:X \to X$ are homotopic maps, then $T(X,f)$ and $T(X,g)$ are homotopy equivalent.
\end{lemma}
\begin{proof}
Let $H$ be a homotopy from $f$ to $g$. Define $F:T(X,f)\to T(X,g)$ and $G:T(X,g)\to T(X,f)$ by
\[
F(x,s) = \begin{cases} (x,2s) \qquad &\text{if  $s\le \frac{1}{2}$}\\
(H(x,2s-1),0) & \text{if  $s\ge \frac{1}{2}$}
\end{cases}\]
\[
G(x,s) = \begin{cases} (x,2s) \qquad &\text{if  $s\le \frac{1}{2}$}\\
(H(x,2-2s),0) & \text{if  $s\ge \frac{1}{2}$}
\end{cases}\]
It is now easy to see that $F\circ G$ and $G\circ F$ are homotopic to the respective identities by 
writing down explicit homotopies.
\end{proof}

We are now ready to attack the proof of Theorem \ref{tonethree}. 

\begin{proof}[Proof of Theorem \ref{tonethree} one direction]
We assume $X$ is finitely dominated so we have $X\xrightarrow{i} K$ and
$K\xrightarrow{r} X$ with $r\circ i \simeq 1_X$. This means $T(X,r\circ i) \simeq
X\times S^1$ is homotopy equivalent to $T(K,i\circ r)$ which is a finite $\cw$-complex. 
Now since $X\times S^1$ is homotopy equivalent to a $\cw$-complex, so is the cyclic cover
$X\times \br$ which in turn is homotopy equivalent to $X$. From now on we replace $X$ by a $\cw$-complex. 
We choose a basepoint $x\in X$ and use the image as basepoint in $K$.
 The fundamental group of $T(K,i\circ r)$ is 
finitely presented since $T(K,i\circ r))$ is a finite $\cw$-complex. 
It follows that $\pi_1(X,x)\times \bz$ 
is finitely presented and by adding just one more relation $\pi_1(X,x)$ is finitely presented.

Now denote the map $i\circ r:K\to K$ by $p$. We know that $p\circ p$ is homotopic to $p$. We can assume that $r$ and $i$ preserve 
a basepoint, but we want the homotopy to preserve the basepoint as well, since this will provide us with an equivariant homotopy between 
$\tilde p\circ \tilde p$ and $\tilde p$. For this purpose we need a basepoint preserving 
homotopy from $r\circ i$ to $1_X$. We now choose a special domination namely $X \to X\times S^1 \simeq L$ with the obvious
 retraction. The homotopy equivalence $X\times S^1 \to L$ induces isomorphism of homotopy groups, and hence, since it is a map 
 of $\cw$-complexes, we may use Whitehead's theorem to produce a pointed homotopy equivalence inducing the given isomorphism 
on homotopy groups. This produces a pointed homotopy
between $p$ and $p\circ p$ and hence an equivariant homotopy between $\tilde p$ and $\tilde p \circ\tilde p$.
It is convenient to have the domination induce isomorphism on the fundamental group, but this may be obtained by adding
a $2$-cell to $L$.

Now consider the singular functor sending $[n]$ to 
\[ S_n(\tilde X)= \{\sigma:\Delta^n\to \tilde X\}\]
followed by realization
\[
|S_.(\tilde X)|= \coprod_n S_n(\tilde X)\times \Delta^n/\{\alpha_*\sigma(x)\simeq\sigma(\alpha^*(x))\}\] 
Evaluation gives a continuous map $|S_*(\tilde X)|\to \tilde X$ which is $\pi$-equivariant.
We approximate to a cellular map $\pi$-equivariantly and apply cellular chains on both sides. The cellular chains on the left side  
is the same as the (normalized) singular chain complex 
and the map induces an isomorphism in homology. Since we have chain complexes of free $\bz[\pi]$-modules, the chain complexes are 
chain homotopy equivalent. The proof of this direction is now finished by applying the following lemma due to A. Ranicki \cite{ra1}.
\end{proof}

\begin{lemma} 
Let $R$ be a ring and $A_*$  be a chain complex. Assume there exists a finite length
chain complex of finitely generated projective $R$-modules $C_*$, and chain maps
$i:A_*\to C_*$ and $r:C_* \to A_*$ so that $r\circ i$ is chain homotopic to the identity.
Then $A_*$ is homotopy equivalent to an infinite length chain complex of finitely generated projective modules
\[0\to F_n\to F_{n-1}\to \cdots\to F_1\to F\xrightarrow{P}
F\xrightarrow{1-P} F\xrightarrow{P} F\xrightarrow{1-P}\cdots\]
where $P:F\to F$ is a projection.
\end{lemma}
\begin{remark} It is easy to see this chain complex in turn is chain homotopy equivalent to
\[0\to F_n\to F_{n-1}\to \cdots\to F_1\to \ker(P)\to 0
\]
which is what we have been aiming for. Notice that if the modules $C_i$ are free we get an infinite chain complex 
of free modules, but in the end we still only get a finite length chain complex of projective modules.
\end{remark}
\begin{proof}[Proof of lemma] 

Let $C_*: 0\to C_n\to \cdots\to C_1\to C_0\to 0$ be a dominating complex so
we have $A_*\xrightarrow{i}C_*$ and $C_*\xrightarrow{r}A_*$ and a chain
homotopy $s$ of $A_*$ so $s\bd + \bd s= 1 - r\circ i$. Define $F= C_0\oplus
C_1\oplus \cdots\oplus C_n$. Then
\[P = \begin{pmatrix} ir & \bd & 0 & 0 & 0 & \cdots\\
isr & 1-ir & -\bd & 0 & 0 & \cdots\\
is^2r & -isr & ir & \bd & 0 & \cdots\\
is^3r & -is^2r & isr & 1-ir & -\bd & \cdots\\
is^4r & -is^3r& is^2r & -isr & ir & \cdots\\
\vdots & \vdots & \vdots & \vdots & \vdots
\end{pmatrix} 
\]
has $P^2=P$. Defining
\[F_r = C_r\oplus C_{r+1} \oplus\cdots\oplus C_n\]
and $\bd: F_i \to F_{i-1}$ by
\[\begin{pmatrix} \bd & 0 & 0 & \cdots\\
1-ir & -\bd & 0 & \cdots\\
-isr & ir & \bd & \cdots\\
is^2r & isr & 1-ir & \cdots\\
\vdots & \vdots & \vdots
\end{pmatrix}
\]
we get a chain complex $F_*\quad 0 \to F_n \to \cdots \to F_1\to F
\xrightarrow{P} F\xrightarrow{1-P}F\cdots$
which one may check is homotopy equivalent to $A_*$ by
\[I:A_* \to F_* \qquad\text{and}\qquad R:F_*\to A_*\]
where
\[I = \begin{pmatrix} i\\ is \\ is^2\\ \vdots
\end{pmatrix}  : A_m \to F_m\]
and $R= (r,0,0,\ldots): F_m\to A_m$. Since $R\circ I = r\circ i \sim 1_A$
and the projection $F_r = C_r\oplus C_{r-1}\oplus \cdots \oplus C_n \to
F_{r+1} = C_{r+1} \oplus\cdots\oplus C_n$ is a homotopy from $IR$ to the
identity, we are done.
\end{proof}
We now consider the proof of the other direction of Theorem \ref{tonethree}. For this we need some lemmas.

\begin{lemma}
Given a pair of topological spaces $(X,Y)$ with $\pi_1(X,y)\cong\pi_1(Y,y)=\pi$ and a 
collection of elements $a_i\in\pi_k(X,Y), k\ge 2$ represented by 
$\alpha_i:(D^k,S^{k-1})\to (X,Y)$. 
Let $Z=\bigcup_i Y\cup _{\alpha_i}D^k$ and extend the map from $Y\to X$ to $Z\to X$ using $\alpha_i$. Then 
$H_l(\tilde X,\tilde Z)\cong H_l(\tilde X, \tilde Y)$ for $l\ne k, k+1$, $H_k(\tilde Z,
\tilde Y)=\oplus_i\bz[\pi]$ and we have an exact sequence
\[  0\to H_{k+1}(\tilde X,\tilde Y)\to H_{k+1}(\tilde X, \tilde Z)\to H_k(\tilde Z,\tilde Y)
\to H_k(\tilde X,\tilde Y)\to H_k(\tilde X,\tilde Z)\to 0
\]
and the map 
\[ \oplus_i\bz[\pi]\cong H_k(\tilde Z,\tilde Y)\to H_k(\tilde X,\tilde Y)
\]
is given by $\pi_k(X,Y,x)\cong \pi_k(\tilde X,\tilde Y,\tilde x)$ followed by the Hurewicz homomorphism applied to the $\alpha_i$.
\end{lemma}

Notice that if $k=2$, we have $\alpha_i:S^1\to Y$ is null homotopic since we have isomorphic fundamental groups. 

\begin{proof}
Use the long exact sequence of the triple $(\tilde Z,\tilde Y)\to (\tilde X,\tilde Y)\to
(\tilde X,\tilde Z)$ and excision on the pair $(\tilde Z,\tilde Y)$.
\end{proof}

\begin{lemma} Let $0\to P_n\to P_{n-1}\to \ldots \to P_2\to P_1\to P_0\to 0$ be a chain 
complex of finitely generated projective $R$-modules. Assume $H_i(P_*)=0$ for $i\le k$. Then 
$P_*$ is chain homotopy equivalent to a chain complex of finitely generated projective $R$-modules
of the form $0\to P_n\to \ldots P_{k+2}\to Q\to 0$. In particular $H_{k+1}(P_*)$ is finitely generated.
\end{lemma}
\begin{proof} Assuming $k\ge 0$ we have $P_1\to P_0$ is onto and by projectivity we may split 
the map and write $P_1=P_0\oplus Q_1$ and the map projection on the first summand. If $k>0$ we may repeat this argument until we meet a non-vanishing homology group.  
\end{proof}
We are now ready for
\begin{proof}[Proof of Theorem \ref{tonethree} other direction]
Since $\pi_1(X,x)$ is finitely presented we can construct a connected pointed $2$-complex $K$ and a pointed map $K\to X$ inducing 
isomorphism on $\pi_1$. The Hurewicz theorem shows that $\pi_2(\tilde K,\tilde k)\to H_2(\tilde K)$
and $\pi_2(\tilde X,\tilde x)\to H_2(\tilde X)$ are isomorphisms. It now follows from the five lemma that 
$\pi_2(\tilde X,\tilde K,\tilde k)=H_2(\tilde X,\tilde K)$. It follows from
 the lemmas above that this is a finitely 
generated $\bz[\pi]$-module since it is the homology of the mapping cone of 
$C_*(\tilde K) \to P_*$, a finite length chain complex of finitely generated projective modules. Choose a generating set 
$\alpha_i:(D^2, S^1) \to (X,K)$ and attach cells to obtain a space $L$, and extend the map. It follows from the lemma above
 that $H_i(\tilde X,\tilde L)=0$ for $i\le 2$. We now continue using the relative Hurewicz theorem to attach cells to obtain a space 
we shall again denote as $L$ so that $H_i(\tilde X,\tilde L)=0$ for $i\ne k$ and $H_k(\tilde X,\tilde L) = P$ a projective 
$\bz[\pi]$-module.
Now consider $(\widetilde{X\times S^1}, \widetilde{L\times S^1})$. We have not changed the homology since the universal 
cover of $S^1$ is contractible. Choose a projective module $Q$ so that $P\oplus Q$ is free and consider the short exact sequence 
\[ 0\to P\otimes \bz[t,t^{-1}] \oplus Q\otimes \bz[t,t^{-1}] \xrightarrow{(1-t)\oplus 1}
P\otimes \bz[t,t^{-1}] \oplus Q\otimes \bz[t,t^{-1}] \to P \to 0
\]
The two first terms are finitely generated free $\bz[\pi][t,t^{-1}]$-modules, so we may first attach $k$-cells to 
$L\times S^1$ and then $(k+1)$-cells to obtain a finite complex which is homotopy equivalent
to $X\times S^1$ since the relative homology of the universal cover is trivial. This shows that $X$ is finitely dominated 
namely by a finite complex homotopy equivalent to $X\times S^1$ 
\end{proof}
We now turn to the proof of theorem \ref{tonetwo}
\begin{proof}[Proof of Theorem \ref{tonetwo}]
We continue the argument from the proof above and assume we have constructed a finite complex
$L$ and a map $L\to X$ so that $\pi_1(L,l)\to \pi_1(X,x)$ is an isomorphism and $H_l(\tilde X,\tilde L)$ is $0$ for $l\ne k$ and 
$H_k(\tilde X,\tilde L)$ is a finitely generated projective $\bz[\pi]$-module which we denote by $P$. Choose a finitely generated 
projective module $Q$ so that $P\oplus Q$ is free. Consider the short exact sequence 
\[0\to Q\oplus P\oplus Q\oplus P \oplus Q\ldots \to P\oplus Q\oplus P\oplus Q\oplus P\ldots \to P\to 0\]
where the map is given by shifting by one. The first two terms are free and we can use a basis to attach infinitely many $k$-cells 
corresponding to the middle term, and 
then infinitely many $(k+1)$-cells corresponding to the first term to obtain a finite dimensional $\cw$-complex so the relative 
homology of the universal cover is $0$ hence $X$ is homotopy equivalent to a finite dimensional $\cw$-complex.

To see $X$ is homotopy equivalent to an infinite dimensional $\cw$-complex with finite skeleta consider the short exact sequence
\[ 0\to Q\to P\oplus Q\to P\to 0
\]
The middle term is free, so we may attach $k$-cells to obtain a complex so the relative homology of the universal cover is $Q$ in dimension $k+1$. Then consider the short exact sequence
\[0\to P\to P\oplus  Q\to Q\to 0\]
and attach $(k+1)$-cells. Continuing this process ad infinitum we only attach finitely many cells in each dimension to obtain an 
infinite dimensional complex with the relative homology of universal covers $0$, hence homotopy equivalent to $X$.

We have already shown $X\times S^1$ is homotopy equivalent to a finite $\cw$-complex using Mather's trick, and we have shown it again
 above if we used the alternative definition of finitely dominated suggested by theorem 1.3. The element $0 \in \tilde K_0(\bz[\pi])$ is 
represented by a finitely generated free module so if $X$ is homotopy equivalent to a finite complex $K$ we see that $S_*(\tilde X)$
is homotopy equivalent to the cellular chains of $\tilde K$ and hence $\sigma(X)=0\in \tilde K_0(\bz[\pi])$.

Now we need to show that if $\sigma(X)=0$ then $X$ is homotopy equivalent to a finite complex. Constructing $L$ as above
it is easy to see that $\sigma(X)= \pm [P]$ according to whether $k$ is even or odd. If $\sigma(X)=0$ this means that $P$ 
is stably free, so we may construct a 
short exact sequence
\[0\to E\to F\to P\to 0\]
where $E$ and $F$ are finitely generated and free. We now attach finitely many $k$-cells corresponding to generators of $E$ and 
finitely many $(k+1)$-cells corresponding to generators of $F$ to obtain a finite complex homotopy equivalent to $X$.

Finally  let us show that for every element $\sigma \in \tilde K_0(\bz[\pi])$ there exists a finitely dominated space $Y$ with 
$\sigma(Y)=\sigma$. We may represent $\sigma$ by a finitely generated projective module $P$, and choose a finitely generated $Q$ so that
 $P\oplus Q$ is finitely generated free. Now choose a finite complex $K$ with $\pi_1(K,k)=\pi$. Let $L$ be $K$ wedge a countably infinite 
wedge of $k$ spheres $k\ge 2$. We now have $H_*(\tilde L,\tilde K)$ is $0$ except for $*=k$ and we may identify 
$H_k(\tilde L,\tilde K)$ with an infinite sum of $Q\oplus P$. Use the short exact sequence 
\[ 0\to Q\oplus P\oplus Q \ldots \to P\oplus Q\oplus P \oplus Q \ldots \to P\to 0\]
to attach $(k+1)$-cells to $L$ to obtain $Y$ such that $H_*(\tilde Y,\tilde K)$ is $0$ except for one dimension where it is $P$. 
Crossing with $S^1$ we may argue as above to show $Y\times S^1$ is homotopy equivalent to a finite complex, hence $Y$ is finitely 
dominated, and clearly $\sigma(Y)=\sigma$.

This ends the proof of Theorem 1.2.
\end{proof}
\section{Examples and conjectures}
It is natural to consider which finiteness obstructions may be realized for various classes of spaces. Here we give some examples of 
that and some conjectures. Recall an $H$-space is a pointed space $X$ together with a map $X\times X\to X$ so that the composite
\[ X\vee X \subset X\times X\to X\]
is homotopic to the folding map. Examples are topological groups and loop spaces.

\begin{conjecture} Let $X$ be a finitely dominated $H$-space, then $\sigma(X) = 0$.
\end{conjecture}
For $H$-spaces that are $\cw$-complexes G. Mislin \cite{mislin1} has shown that $X$ is finitely dominated if and only if $\oplus H_i(X,\bz)$ is a finitely 
generated abelian group.
The conjecture above has been verified in a number of cases.
\begin{proposition} If $\pi_1(X,x)$ is infinite, then $\sigma(X)=0$.
\end{proposition}
\begin{proof} Being an $H$-space $\pi_1(X,x)$ is abelian, so being infinite $\pi_1(X,x)\cong \bz\oplus A$. Now let $S^1\to X$ represent 
a generator of $\bz$ and use the projection $\bz\oplus A\to \bz$ to define an element in $H^1(X,\bz)$ which in turn may be 
represented by a map $X\to S^1$. The composite $S^1\to X\to S^1$ is homotopic to the identity. Let $Y$ be the homotopy fibre of $X\to S^1$
and consider 
\[ Y\times S^1\to X\times X \to X.\]
This induces an isomorphism of homotopy groups so $X\simeq Y\times S^1$. We clearly have $Y$ is finitely dominated so $\sigma(X)=0$ by Theorem 1.2.
\end{proof}

Another case where the conjecture has been verified is when $X$ has a classifying space i.~e. $X\simeq \Omega B$ \cite{bknp1}.

\begin{theorem} If $X\simeq \Omega B$ for some $\cw$-complex $B$, and $\oplus H_i(X,\bz)$ is a finitely generated abelian group, 
then $X$ is homotopy equivalent to a smooth, compact, parallellizable manifold.
\end{theorem}

The first step in the proof of this theorem is to deal with the finiteness obstruction.

Another example is finitely dominated nilpotent spaces \cite{elp1} . In this case there is a group between $0$ and $\tilde K_0(\bz[\pi])$ which describes precisely which 
elements can be realized by a nilpotent space with fundamental group $\pi$. There are examples where this group is strictly bigger than $0$ and strictly 
smaller than $\tilde K_0(\bz[\pi])$.

Another conjecture concerns classifying spaces of groups.

\begin{conjecture} Let $\Gamma$ be a group, $B\Gamma$ its classifying space. Assume $B\Gamma$ is finitely dominated, then $\sigma(B\Gamma)=0$
\end{conjecture}

This conjecture has also been verified in many cases since it has been proved in many cases that $\tilde K_0(\bz[\Gamma])=0$. This is 
a consequence of the Hsiang conjecture, see  W.-C. Hsiang \cite{hs1}, 
which says that the assembly map
\[B\Gamma_+\wedge K(\bz) \to K(\bz[\Gamma])\]
is a homotopy equivalence of spectra for $\Gamma$ torsion free. Here $K(-)$ is non-connective $K$-theory. The Hsiang conjecture has been verified in a large number of cases by Farrell, Jones, L\"uck, Bartels, 
Reich, and others.

Finally it is natural to mention the first case where the finiteness obstruction occurred at least morally, namely in R. Swan's work \cite{sw2} on $\cw$-complexes with universal cover 
homotopy equivalent to a sphere. Here Swan constructed many examples that were finitely dominated, but with non-trivial finiteness obstruction.

\providecommand{\bysame}{\leavevmode\hbox to3em{\hrulefill}\thinspace}
\providecommand{\MR}{\relax\ifhmode\unskip\space\fi MR }
\providecommand{\MRhref}[2]{%
  \href{http://www.ams.org/mathscinet-getitem?mr=#1}{#2}
}
\providecommand{\href}[2]{#2}

\end{document}